\documentclass[10pt]{amsart}
\usepackage[utf8]{inputenc}
\usepackage[english]{babel}
\usepackage{amsmath}
\usepackage{amsfonts}
\usepackage{amssymb}
\usepackage{amsthm}
\usepackage{siunitx}
\usepackage{bbm}
\usepackage[colorlinks,allcolors=blue]{hyperref}
\usepackage[square,numbers]{natbib}
\usepackage{comment}
\usepackage{tasks}
\usepackage{enumitem}
\usepackage{stix}
\usepackage{multicol}
\usepackage{mathtools}
\usepackage{graphicx}
\usepackage{geometry}
\usepackage{mathrsfs}
\usepackage{pgf,tikz,pgfplots}

\newtheorem{theorem}{Theorem}

\newtheorem{definition}{Definition}
\newtheorem{corollary}{Corollary}
\newtheorem{example}{Example}

\DeclareMathOperator*{\esssup}{ess\,sup}
\usetikzlibrary{calc}
\usetikzlibrary{arrows}
\pgfplotsset{compat=1.15}

\title{Counting Problem for Some Random Conformal Iterated Function Systems}
\author{Hamid Naderiyan\\  University of California - Riverside}

\begin{document}

\begin{abstract}
    This paper studies the counting problem in random dynamical systems. We noticed that the nature of counting in the random setting is completely different than that of the deterministic systems in the sense that non-exponential growth is constructed on a set of full measure. The recurrent behavior of random walks plays a major role in counting in the random setting.
\end{abstract}

\maketitle

\tableofcontents
\begin{large}

\section{Introduction}
The author of this article formerly investigated the counting problem in the context of symbolic spaces for counting finite words subject to a certain length condition, see \cite{HN}. The motivation of the current work is a question posed by M. Urba\'{n}ski for the possibility of finding any asymptotic formula for the growth of an analogous counting function when the deterministic dynamical system is replaced by a random dynamical system, see the equations \ref{eq1} and \ref{random counting} for comparison. In the deterministic setting, this problem has been investigated in various forms. For the references, one can check the introduction of the author's other paper \cite{HN}.

Our main focus in this paper is the context of random conformal iterated function systems (random CIFS). There is no reason one restrict themselves only to this context. There are certain ways that the counting problem in the random setting resembles similar behavior to that of the deterministic setting. Except for one example, we did not mention such random systems in this paper. In fact, the existence of such random systems suggests that there may not be a universal answer to the counting problem in the random setting unless expressed in the language of probability theory.

The study of the counting problem in a random setting faces several challenges. Naturally, one starts with an approach similar to the existing methods for the deterministic setting, which is studying the boundary behavior of the Laplace–Stieltjes transform (Poincar\'{e} series) of the counting function and then applying an appropriate Tauberian theorem for finding a good asymptotic formula for the counting function. However, the understandable behavior of the deterministic Poincar\'{e} series is due to the nice property of the transfer operator having a simple eigenvalue. But for the random transfer operator, this property is not even well-defined in general, since the domain and range of the random transfer operator may not necessarily be identical. Even in the cases where the domain and range are identical, the challenge is the iteration pattern of the random transfer operator. In fact, there is a map that the iteration of the random transfer operator has to follow its pattern. Therefore, other than some special cases the traditional tools in spectral theory and Tauberian theory are not enough for the random setting. The other issue is the very definition of random dynamical systems. For instance, we restricted ourselves only to random CIFS in dimensions one and two, i.e. when the functions are defined over a subdomain of $\mathbb{R}^1$ and $\mathbb{R}^2$, because we faced challenges for higher dimensions, due to the very definition of the random system that made the counting problem not interesting enough. Other than that, there are no limitations concerning the dimension of the system under consideration. Also, we did not consider situations in which the random factor is bigger than one, because it may cause some issues with the very definition of random CGDMS.

The main result of this article is theorem \ref{main theorem 4} in which we show that the counting problem can be very exotic in a random setting. This result partially answers one of Urba\'{n}ski's questions. That the limit infimum and limit supremum can be zero and infinite correspondingly. We obtained a nonzero limit infimum and a finite limit supremum when the counting is subject to certain conditions that resemble periodic behavior within the random dynamics, see theorem \ref{future periodic d-generic}. We also considered a more general case in which one can have non-zero limit infimum and finite limit supremum.

Concerning the structure of the paper, we first mention some preliminary facts and some background results from dynamical system in section \ref{pre}. Then we discuss the random dynamical system setting for which we want to investigate the counting problem in section \ref{sec 3}. In fact, we formulate the counting problem in this section. In section \ref{exp growth} we state our results when the growth of the counting function is exponential. Next, we state our major result in section \ref{non exp growth} concerning non-exponential growth. Finally, we finish with some examples to bring more details on the subject in the last section.

\section{Preliminaries}\label{pre}

\noindent We start by considering $ E$ a countable (finite or infinite) set of symbols (letters). The symbolic space $E^{\mathbb{N}}$ (all called full shift space) is the set of all infinite sequences (one-sided) of the form 
$$e_1e_2e_3...e_n...,$$
where each $e_i \in E$. We usually represent the $n$ consecutive symbols of such a sequence by $\omega$, i.e. 
$$\omega=e_1e_2...e_n,$$
and we call $\omega$ a word. By $|\omega|=n$, we mean the word $\omega$ has length $n$, i.e. it contains $n$ letters. By $E^n$ we consider all the words of length $n$ and $E^*$ represents the set of all words of any length, i.e. $E^*=\cup_{n=1}^{\infty} E^n$. In addition, with an incidence matrix $A: E\times E \rightarrow \{0,1\}$ we can prohibit some sequences. In fact a sequence is permitted if and only if
$$A_{e_1e_2}:=A(e_1,e_2)=1, \; A_{e_2e_3}=A(e_2,e_3)=1, \; ..., \; A_{e_ne_{n+1}}=A(e_n,e_{n+1})=1, \; ... \;.$$
By a subscript notation $A$, we emphasize that only some sequences or words are considered. In fact, $E_A^{\mathbb{N}}$ (also called subshift) denotes the set of all permitted (or admissible) sequences. By $E^n_A$ we represent all the admissible words of length $n$. $E^*_A$ represents all the admissible finite words of any length. Also for a $\rho \in E_A^{\mathbb{N}}$, the set $E^*_{\rho}$ contains all $\omega \in E^*_A$ such that $\omega \rho$ is an admissible sequence. By $E^n_{\rho}$ we mean all $\omega \in E^n_A$ such that $\omega \rho$ is an admissible sequence. By $E^*_{\text{per}}$ we mean all $\omega \in E^*_A$ such that $\omega_n \omega_1$ is admissible. In this case, we say $\omega$ is \textbf{periodic word} and by $\bar{\omega}$ we represent the sequence $\omega \omega\omega ...$. Finally, for each finite word $\omega$ of length $n$ we define the \textbf{cylinder} 
$$[\omega]:=\{\rho \in E_A^{\mathbb{N}} : \; \rho_1...\rho_n=\omega  \}.$$
All the cylinders form a base for a topology on $E_A^{\mathbb{N}}$.
Also, the shift map on $E_A^{\mathbb{N}}$ is defined by 
$$\sigma: E_A^{\mathbb{N}} \to E_A^{\mathbb{N}} \; \; \; \; \; \; \sigma(e_1e_2e_3e_4...)=e_2e_3e_4... \; .$$
We call a subshift \textbf{finitely irreducible} if there exists a finite set $\Omega$ containing words such that for all $e,e' \in E$ there is $\omega \in \Omega$ such that $e\omega e'$ is admissible. Throughout this paper, we restrict ourselves to work with finitely irreducible subshifts.

A complex-valued function $f$ on $E_A^{\mathbb{N}}$ is called \textbf{summable} if 
$$\sum_{e \in E} \exp(\sup_{[e]} \operatorname{Re} (f))<\infty.$$
For a complex-valued H\"{o}lder-type (see \cite{HN} for definition) summable function $f$ we introduce \textbf{transfer} (Ruelle-Perron-Frobenius) operator:
$$\mathcal{L}_{f}:C_b(E_ A^{\mathbb{N}},\mathbb{C}) \rightarrow C_b(E_ A^{\mathbb{N}},\mathbb{C})$$ $$\mathcal{L}_{f} (g)(\rho)=\sum\limits_{\substack{ e \in E_{\rho}^1}}\exp \left(f(e\rho)\right)g(e\rho).$$
The \textbf{topological pressure} of $xf$ (where $x$ is a real number and $f$ is a real-valued function on $E_A^{\mathbb{N}}$) is defined by
$$P(x)=P(xf)=\lim_{n\to \infty}\frac{1}{n}\ln \big( \sum\limits_{\omega \in E_A^n}\exp( x\sup_{[\omega]}S_nf) \big),$$
where $S_nf$ is the \textbf{ergodic sum} of $f$ given by
$$S_nf(\rho)=\sum_{i=0}^{n-1}f(\sigma^i(\rho)).$$
The pressure function is strictly decreasing on its domain:
$$\Gamma=\{ x \in \mathbb{R}: xf \; \; \text{summable} \}.$$
A real-valued function $f: E^{\mathbb{N}}_A \rightarrow \mathbb{R}$ is called \textbf{regular} if $P(x)=0$ for some $x>0$ and is called \textbf{ strongly regular} if it is regular and $0<P(x)<\infty$ for some $x>0$. If the pressure function $P$ has any zero, it will be a unique zero of $P$. We represent this zero by $\delta$.
Note that strong regularity implies
$$\inf \Gamma < \delta.$$
We restrict ourselves to working with strongly regular systems in this paper.
We also want to emphasize some quick tips about the spectral analysis on the transfer operator. For more explanations, we refer to \cite{HN}.
\begin{itemize}
    \item $\mathcal{L}_{s}:=\mathcal{L}_{s f}$ is an operator on $C^{0,\alpha}(E_A^{\mathbb{N}},\mathbb{C})$ for any  $s \in \Gamma ^+=\Gamma \times \mathbb{R} \subseteq \mathbb{C}$.
    \item For every $n\in \mathbb{N}$, the operator-valued function $s \mapsto \mathcal{L}_s ^n$ is holomorphic on $\Gamma ^+$.
    \item The spectral radius of $\mathcal{L}_s$ is at most $e^{P(x)}$ and the essential spectral radius of $\mathcal{L}_{s}$ does not reach $e^{P(x)}$, i.e. $r(\mathcal{L}_s) \leq e^{P(x)}$ and  $r_{\text{ess}}(\mathcal{L}_s)<e^{P(x)}$.
    \item The transfer operator $\mathcal{L}_s$ has at most finitely many eigenvalues of modules $e^{P(x)}$ all of which with multiplicity one.
    \item The spectral representation of $\mathcal{L}_s$ corresponding to the eigenvalues $\xi_1, \xi_2, ...,  \xi_p$ of modulus $e^{P(x)}$ is of the form:
$$ \mathcal{L}_s=\xi_1(s)\mathcal{P}_{1,s}+ \xi_2(s)\mathcal{P}_{2,s} + ...+ \xi_n(s)\mathcal{P}_{n,s}+\mathcal{D}_{s},$$
where each $\mathcal{P}_{i,s}$ is a projection operator.
\item The $m^{th}$ iteration of $\mathcal{L}_s$ is of the form
\begin{equation}\label{spectral decomposition sum}
\mathcal{L}_s^m=\xi_1(s)^m\mathcal{P}_{1,s}+ \xi_2(s)^m\mathcal{P}_{2,s} + ...+ \xi_n(s)^m\mathcal{P}_{n,s}+\mathcal{D}_{s}^m.
\end{equation} 
\end{itemize}
Another important property of the transfer operator to be discussed is the D-generic property. This property prohibits the possibility of admitting specific eigenvalues. We say $f$ is \textbf{D-generic} if $\mathcal{L}_{x+iy}: C^{0,\alpha}(E_A^{\mathbb{N}},\mathbb{C}) \rightarrow C^{0,\alpha}(E_A^{\mathbb{N}},\mathbb{C})$ does not admit $\exp(P(x))$ as eigenvalue when $y\neq 0$. Further, we say $f$ is \textbf{strongly D-generic}, if  $\mathcal{L}_{x+iy}: C^{0,\alpha}(E_A^{\mathbb{N}},\mathbb{C}) \rightarrow C^{0,\alpha}(E_A^{\mathbb{N}},\mathbb{C})$ does not admit any eigenvalue of magnitude $\exp(P(x))$ when $y\neq 0$. 

Given $\rho \in E_A^{\mathbb{N}}$, we define a \textbf{counting functions} by
\begin{equation}\label{eq1}
N_{\rho}(T):=\#\{ \omega \in \cup_{n=1}^{\infty} E_A^n \; : \; \omega\rho \text{ admissible},\;  S_{|\omega|}f(\omega\rho)\geq - T \}, \; \; \; \; \; \; T>0
\end{equation}
This is a function of $T$. We also consider the Laplace–Stieltjes transform of $T \mapsto N_{\rho}(T)$ (also called the \textbf{Poincar\'{e} series}):
$$\eta_{\rho}(s):=\int _0 ^{\infty} \exp (-sT)\text{d}N_{\rho}(T).$$
The functions $\eta_{\rho}(s)$ is holomorphic on $\Gamma^+$. For the proof, see \cite{HN}. Now we are ready to present an asymptotic formula for this counting function. For its proof, see \cite{HN}.
\begin{theorem}\label{border 0 non-degeneric}
If $f: E_A^{\mathbb{N}} \to \mathbb{R}$ is strongly regular H\"{o}lder-type function with $P(\delta f)=0$, for every Borel set $B \subseteq E_A^{\mathbb{N}}$ with boundary of measure $0$ and $\rho \in E_A^{\mathbb{N}}$ we have
$$c_{\delta}\dfrac{h_{\delta}(\rho)}{\chi_{\mu}}m(B)\leq \liminf_{T \rightarrow \infty}\dfrac{N_{\rho}(B,  T)}{\exp(\delta T)}\leq \limsup_{T \rightarrow \infty}\dfrac{N_{\rho}(B, T)}{\exp(\delta T)}\leq c_{\delta} \dfrac{h_{\delta}(\rho)}{\chi_{\mu}}m(B)+y_0^{-1} \dfrac{h_{\delta}(\rho)}{\chi_{\mu}},$$
and 
$$c_{\delta}\dfrac{1}{\chi_{\mu}}\mu(B)\leq \liminf_{T \rightarrow \infty}\dfrac{N_{\text{per}}(B, T)}{\exp(\delta T)}\leq\limsup_{T \rightarrow \infty}\dfrac{N_{\text{per}}(B, T)}{\exp(\delta T)}\leq c_{\delta} \dfrac{1}{\chi_{\mu}}\mu(B)+y_0^{-1} \dfrac{1}{\chi_{\mu}},$$
where 
$$c_{\delta}:=y_0^{-1}\left(\exp(\delta y_0^{-1})-1\right)^{-1}, \; \; \; \; y_0>0. $$
\end{theorem}
\noindent Note that $y_0$ is related to the pole(s) of the Poincar\'{e} series. In a special case when the only pole of this series along its critical line of convergence is $s=\delta$, then $y_0 \to \infty$ and we obtain the following corollary.
\begin{corollary}[Pollicott-Urba\'{n}ski]\label{llll}
Let $\mathcal{S}=\{ \phi_e \}_{e \in E}$ be a strongly regular conformal graph-directed Markov system with D-generic property. Let $\delta$ be the Hausdorff dimension of the limit set of $\mathcal{S}$, then for every Borel set $B \subseteq E_A^{\mathbb{N}}$ with boundary of measure $0$ and $\rho \in E_A^{\mathbb{N}}$ we have
$$\lim_{T \rightarrow \infty}\dfrac{N_{\rho}(B, T)}{\exp( \delta T)}=\dfrac{h_{\delta}(\rho)}{\delta \chi_{\mu_{\delta}}}m_{\delta}(B),$$
and 
$$\lim_{T \rightarrow \infty}\dfrac{N_{\text{per}}(B, T)}{\exp(\delta T)}=\dfrac{1}{\delta \chi_{\mu_{\delta}}}\mu_{\delta}(B).$$
\end{corollary}

\noindent Now we want to define the \textbf{random graph-directed Markov system} (random GDMS).  We adopt this notion from Roy-Urba\'{n}ski, see \cite{RU2}. We start with a directed multi-graph $(V, E, i, t)$, a mapping $A: E\times E \rightarrow \{0,1\}$ and a family of compact Euclidean metric spaces $\{X_v\}_v$. Then we employ an invertible ergodic measure preserving map $T: (\Lambda, B, \nu) \rightarrow (\Lambda, \mathcal{F}, \nu)$ on a complete probability space $(\Lambda, B, \nu)$ and a family of injective contractions $\{\phi_e^{\lambda}: X_{t(e)} \rightarrow X_{i(e)}\}_{e \in E,\lambda \in \Lambda}$ with Lipschitz constants at most $\kappa \in (0,1)$. For each word $\omega \in E_A^*$ and each $\lambda \in \Lambda$ we define 
$$\phi_{\omega}^{\lambda}:=\phi_{\omega_1}^{\lambda} \circ \phi_{\omega_2}^{T(\lambda)} \circ ...\circ \phi_{\omega_n}^{T^{n-1}(\lambda)}.$$
Note that for each $(\omega,\lambda) \in (E_A^*, \Lambda)$ the map $t\mapsto \phi_{\omega}^{\lambda}(t)$ is continuous. The map $\lambda \mapsto \phi_{\omega}^{\lambda}(t)$ is measurable for each $(\omega,t) \in (E_A^*, X_{t(\omega)})$ and $(t,\lambda) \mapsto \phi_{\omega}^{\lambda}(t)$ is jointly measurable for each word $\omega \in E_A^*$. Then for every $\lambda \in \Lambda$, $\rho \in E_A^{\mathbb{N}}$ the collection of sets $\{\phi_{\rho_1...\rho_n}^{\lambda}(X_{t(\rho_n)})\}$ is a decreasing sequence of non-empty compact sets whose diameters are bounded by $\kappa^n$, so $$\cap_{n\geq 1} \phi_{\rho_1\rho_2...\rho_n}^{\lambda}(X_{t(\rho_n)})$$ is a singleton and we denote its only element by $\pi^{\lambda}(\rho)$. Therefore for each $\lambda \in \Lambda$, we define the limit set $J^{\lambda}$ by
$$J^{\lambda}=\pi^{\lambda}(E_A^{\mathbb{N}})=\cup_{\rho \in E_A^{\mathbb{N}}} \pi^{\lambda}(\rho).$$
\begin{definition}\label{random conformal gdms}
We call a system \textbf{random conformal graph directed Markov system} (random CGDMS) if the following conditions are satisfied for some $d \in \mathbb{N}$:
\begin{enumerate}
    \item[a)] For every $v \in V$, $X_v$ is a compact connected subset of $\mathbb{R}^d$ and $X_v=\overline{\text{Int}(X_v)}$.
    \item[b)] (Open Set Condition) For almost every $\lambda$ and all different $e, e' \in E$,
    $$\phi_e^{\lambda}\left(\text{Int}(X_{t(e)})\right)\cap \phi_{e'}^{\lambda}\left(\text{Int}(X_{t(e')})\right)=\emptyset.$$
    \item[c)] (Conformality) For every $v \in V$ there is an open connected set $W_v$ containing $X_v$. Also, for almost every $\lambda$ and each $e \in E$, $\phi_e^{\lambda}$ extends to a $C^1$ conformal diffeomorphism from $W_{t(e)}$ into $W_{i(e)}$ with Lipschitz constant bounded by $\kappa$.
    \item[d)] (Bounded Distortion Property) There are two constants $L\geq 1$ and $\alpha> 0$ such that for every $e \in E$ and every $s, t \in X_{t(e)}$
    $$\left|\frac{|(\phi_e^{\lambda})'(s)|}{|(\phi_e^{\lambda})'(t)|}-1\right| \leq L\|s-t\|^{\alpha},$$
    for almost every $\lambda$, where $|(\phi_e^{\lambda})'(t)|$ denotes the scaling factor of the derivative of $(\phi_e ^{\lambda})'$ at $t$.
\end{enumerate}
\end{definition}
\noindent For a random conformal graph directed Markov system, we define the \textbf{random potential} by
$$f: (E_A^{\mathbb{N}},\Lambda) \rightarrow \mathbb{R}, \; \; \; f(\rho,\lambda):=\log \left|(\phi_{\rho_1}^{\lambda})'\left(\pi^{T(\lambda)}(\sigma(\rho))\right)\right|.$$
The \textbf{essential supremum} of such a real-valued function $f$ on a measure space $(\Lambda,\nu)$ is defined by
$$\esssup_{\lambda} f:=\inf \{r: \; f(\lambda)\leq r \; \text{for } \nu \text{-almost all } \lambda \in \Lambda \}.$$
Such a real-valued function $f$ is called \textbf{summable} if 
$$\sum_{e \in E} \exp \left(\esssup_{\lambda} f(e, \lambda)\right) < \infty. $$
\noindent If we are given such a real-valued function $f$ and a real parameter $x$, then $xf$ would also be a real-valued function. Clearly $xf$ is summable if
$$\sum _{e \in E} \esssup_{\lambda} |(\phi_{e}^{\lambda})'|^x  < \infty.$$
We consider $\Gamma$ the set of all $x$ for which $xf$ is summable, where $f$ is the random potential of the random CGDMS. Also, we set $\Gamma^+=\Gamma \times \mathbb{R}$. These notations are similar to those of the deterministic system but in the rest of the paper, we only work with the random system. So whenever we refer to $\Gamma$ or $\Gamma^+$ we mean the ones just described above for a random system. We introduce a \textbf{random transfer operator} here,

$$\mathcal{L}_{f}^{\lambda}:C_b(E_ A^{\mathbb{N}},\mathbb{C}) \rightarrow C_b(E_ A^{\mathbb{N}},\mathbb{C})$$ $$\mathcal{L}_{f}^{\lambda} (g)(\rho)=\sum\limits_{\substack{ e \in E_{\rho}^1}}\exp \left(f(e\rho, \lambda)\right)g(e\rho).$$

In general, the $n^{th}$ iterate of the random transfer operator is not defined because the domain and range of the operator are not the same. Even though the $n^{th}$ iterate of this operator defined above makes sense, we are only interested in the conventional alternative operator below,

$$(\mathcal{L}_{f}^{\lambda})^n :=\mathcal{L}_{f}^{T^{n-1}(\lambda)} \circ ... \circ \mathcal{L}_{f}^{T^2(\lambda)} \circ \mathcal{L}_{f}^{T(\lambda)} \circ \mathcal{L}_{f}^{\lambda}, \; \; \; n \in \mathbb{N}$$

One can see that for $x\in \Gamma$ and almost every $\lambda \in \Lambda$, there exists a unique bounded measurable function $\lambda \mapsto P^{\lambda}(x)=P^{\lambda}(xf)$, also called the \textbf{random pressure function}, and a unique random probability measure $\{m^{\lambda}_x\}$ such that 
$$(\mathcal{L}_{xf}^{\lambda})^*m^{T(\lambda)}_x=e^{P^{\lambda}(x)}m^{\lambda}_x,$$
for almost every $\lambda \in \Lambda$, see \citep[~ p. 271]{RU2}. That means $P^{\lambda}(x)$ and $m_x$ are uniquely determined by
$$m^{\lambda}_x([e\omega])=e^{-P^{\lambda}(x)}\int_{[\omega]}\left|(\phi_{e}^{\lambda})'\left(\pi^{T(\lambda)}(\tau)\right)\right|^x \text{d}m^{T(\lambda)}_x(\tau),$$
for almost every $\lambda \in \Lambda$. Also, the expected pressure function is defined by 
$$\mathcal{E}P(x):=\int_{\Lambda} P^{\lambda}(x) d\nu.$$
Moreover, we have the following equation from \citep[287]{RU2},
$$\delta_{\Lambda}:=\text{HD}(J^{\lambda})=\inf \{x \; : \; \mathcal{E}P(x)\leq 0 \},$$
for almost every $\lambda \in \Lambda$.

\noindent We would like to mention a few words about the definition of random subshift of finite type. Bogenschutz defines it the way that each fiber is closed in a compact set \citep[p. 420]{Bogenschutz}. Roy-Urba\'{n}ski generalized it to a special case of subshift of finite type with infinite letters \citep[p. 420]{RU2}. Bogenschutz as well uses the infinite letter system $\mathbb{Z_+}$ but he makes the system compact (using one-point compactification $\mathbb{\bar{Z}_+} = \mathbb{Z_+} \cup \{ \infty \}$) so that his bundle random dynamical system theory (reflected in his Ph.d dissertation) applies and this does not include Roy-Urba\'{n}ski random system as special case simply because $\prod_{i = 0}^{\infty} \mathbb{Z_+}$ is not a closed subset of $\prod_{i = 0}^{\infty} \mathbb{\bar{Z}_+}$.

\noindent We also want to mention the \textbf{law of the iterated logarithm} in dimension one. Let $X_1, X_2, X_3, ...$ be a sequence of (real) independent identically distributed random variables whose mean and variance are $\mu$ and $\sigma ^2$ correspondingly. Then 

\begin{equation}\label{lil}
    \liminf_{n \to \infty} \dfrac{X_1+...+X_n -n\mu}{\sqrt{2\sigma ^2 n\log \log n}}=-1, \; \; \; \; \;  \liminf_{n \to \infty} \dfrac{X_1+...+X_n -n\mu}{\sqrt{2\sigma ^2 n\log \log n}}=1, \; \; \; \; \; \text{a.s.}
\end{equation}
For the details see section 4.4 in \cite{revesz}.

\section{Counting in Random Dynamics}\label{sec 3}
\noindent In the previous section, we mentioned two main results for the counting problem in the deterministic dynamical systems, Theorem \ref{border 0 non-degeneric} and corollary \ref{llll}. In this section, we want to formulate the problem in the random dynamical systems. Our focus will be the random conformal iterated function system (CIFS) setting. The notion of random system is what we adopt from Roy and Urba\'{n}ski, see section \ref{pre} or \citep{RU2}. For this purpose, we need to specify a complete probability space with an invertible ergodic measure preserving transformation.

\noindent We consider a countable (finite or infinite) set of complex numbers $z$ within the unit disk that are bounded away from $0$:
$$\mathcal{Z} \subset \{z \in \mathbb{C} : 0<\epsilon <|z|\leq 1 \},$$
and we set:
$$\Lambda:=\mathcal{Z}^{\mathbb{Z}},$$
which is the set of all two-sided sequences with letters in $\mathcal{Z}$. We represent an element of $\Lambda$ by $\lambda$ and of course $\lambda_i$ is $i^{th}$ coordinate of $\lambda$, where $i \in \mathbb{Z}$. The readers will notice later that all results in this paper only depend on the future part of each $\lambda \in \Lambda$, i.e. only on $\lambda_0 \lambda_1 \lambda_2 ...$. But we still require a two-sided sequence $\lambda$ to meet the definition of the random system, see \citep[264]{RU2}. For the $\sigma-$Algebra $\mathcal{B}$ of measurable sets we just consider the Borel sets, for the ergodic invertible measure $\nu$ we consider a Bernoulli probability measure, and for $T$ we consider the shift map on $\Lambda$. Therefore $(\Lambda, \mathcal{B},\nu,T)$ is just two-sided full shift space with an ergodic measure. To summarize, we have:
\begin{itemize}
    \item $\Lambda=\{\lambda= ...\lambda_{-(i-1)}...\lambda_{-2}\lambda_{-1}\lambda_0\lambda_1\lambda_2...\lambda_{i-1}... \; : \; \lambda_i \in \mathcal{Z}, \; i \in \mathbb{Z} \}$
    \item $[\lambda_{i_1}=z_1, \lambda_{i_2}=z_2, ..., \lambda_{i_k}=z_k]=\{\lambda \in \Lambda \; : \; \lambda_{i_1}=z_1, \lambda_{i_2}=z_2, ..., \; \lambda_{i_k}=z_k \}$
    \item $\sum_{z \in \mathcal{Z}}\nu([\lambda_0=z])=1$
    \item $T(...\lambda_{-(i-1)}...\lambda_{-2}\lambda_{-1}\boldsymbol{\lambda_0}\lambda_1\lambda_2...\lambda_{i-1}...)=...\lambda_{-(i-2)}...\lambda_{-1}\lambda_{0}\boldsymbol{\lambda_1}\lambda_2\lambda_3...\lambda_{i}...$
    \item $\nu(T^{-1}(B))=\nu(B), \; \; \; B\in \mathcal{B}.$
\end{itemize}
With this measurable system $(\Lambda, \mathcal{B},\nu,T)$, we can introduce a random system for any deterministic CIFS in dimension one or two. This process can be defined in different ways. We explain an important one, which will be our main focus in this section. Assume $\{ \phi_e \}_{e \in E}$ is a CIFS, i.e. a countable (finite or infinite) family of conformal contractions defined on a compact Euclidean space $X$, then we define
$$\phi_e^{\lambda}:=\lambda_0\phi_e.$$
Note that we have to be careful with our definition so that we make sure it satisfies definition \ref{random conformal gdms}. For instance, we need to know that $\lambda_0$, chosen non-real, only makes sense when $\phi_e$ is a complex-valued function, or this $\lambda_0$ should be appropriate enough that the image of $\phi_e^{\lambda}$ is still in $X$. As well note that this random CIFS is summable exactly for those $x$ that the deterministic system is summable. Therefore with no change, we use the same notation $\Gamma$ for all these $x$ and $\Gamma^{+}$ for the right half-plane. Next, we want to find the random pressure and the random transfer operator associated with this random CIFS. Referring to below the definition \ref{random conformal gdms}, we can obtain our random potential function
$$f(\rho,\lambda)=\log \left|(\phi_{\rho_1}^{\lambda})'\left(\pi^{T(\lambda)}(\sigma\rho)\right)\right|=\log\left|(\phi_{\rho_1})'\left(\pi(\sigma\rho)\right)\lambda_0\right|=\log\left|(\phi_{\rho_1})'\left(\pi(\sigma\rho)\right)\right|+\log|\lambda_0|,$$
where $f(\rho)=\log\left|(\phi_{\rho_1})'\left(\pi(\sigma\rho)\right)\right|$ is just the potential obtained from the deterministic system $\{ \phi_e \}_{e \in E}$, see below the definition 15 in  \cite{HN}. Therefore, for the \textbf{random ergodic sum} we get 
$$S_nf(\rho,\lambda)=f(\rho, \lambda)+f(\sigma \rho, T\lambda)+...+f(\sigma^{n-1}\rho, T^{n-1}\lambda)$$
$$=f(\rho)+\log|\lambda_0|+f(\sigma \rho)+\log|\lambda_1|+f(\sigma^2\rho)+\log|\lambda_2|+...+f(\sigma^{n-1}\rho)+\log|\lambda_{n-1}|$$
\begin{equation}\label{random ergodic sum}
=S_nf(\rho)+\log|\lambda_0\lambda_1...\lambda_{n-1}|.
\end{equation}
We can also calculate the random pressure function:
$$P^{\lambda}(x)=P(x)+x\log |\lambda_0|,$$
for almost all $\lambda$. Therefore the expected pressure will be
\begin{equation}\label{eqr 8} 
\mathcal{E} P(x)=\int_{\lambda}P^{\lambda}(x)\text{d}\nu=P(x)+\left(\sum_{z \in \mathcal{Z}}\nu([\lambda_0=z])\log |z|\right)x
\end{equation}
Note that if there is at least one $z \in \mathcal{Z}$ with $|z|<1$, then the parenthesis above would be a negative value, otherwise, it vanishes. Therefore,
$$\delta_{\Lambda}\leq \delta.$$
\noindent When we say a random CIFS constructed above has D-generic property, we mean the deterministic potential function has the D-generic property. Similarly, for strongly D-generic property.
Furthermore, the random transfer operator is given by
$$\mathcal{L}_x^{\lambda}(g)(\rho):=\mathcal{L}_{xf}^{\lambda}(g)(\rho)=\sum\limits_{e \in E_{\rho}^1}\exp \left( xf(e\rho, \lambda) \right)g(e\rho),$$
for a bounded continuous function $g$ and $x \in \Gamma$. This yields
$$(\mathcal{L}_x^{\lambda})^n(\mathbbm{1})(\rho)=\sum\limits_{\omega \in E_{\rho}^n}\exp (xS_nf(\omega\rho, \lambda))=\sum\limits_{\omega \in E_{\rho}^n}\exp (xS_nf(\omega\rho))|\lambda_0\lambda_1...\lambda_{n-1}|^x$$
$$=|\lambda_0\lambda_1...\lambda_{n-1}|^x\mathcal{L}_x^n(\mathbbm{1})(\rho),$$
where $\mathbbm{1}$ is just the characteristic function of $E_A^{\mathbb{N}}$.
It is important to note that the above expression is not the $n^{th}$ iteration of the operator $\mathcal{L}_x^{\lambda}$, simply because of the random variable $\lambda$ that is involved. Now we are ready to define an appropriate \textbf{counting function} for a fixed $\rho \in E_A^{\mathbb{N}}$ and a fixed $\lambda \in \Lambda$:
\begin{equation}\label{random counting}
   N_{\rho}^{\lambda}(T):= \#\{ \omega \in E_{\rho}^* : S_{|\omega|}f(\omega \rho,\lambda)\geq -T \}, \; \; \; \; \; \; T>0
\end{equation}
It is encouraged to compare this with \ref{eq1}. The Laplace–Stieltjes transform of this counting function (also called the \textbf{random Poincar\'{e} series}) is:
\begin{equation}\label{eta}
\eta_{\rho}^{\lambda}(s):=\int_0 ^{\infty} \exp(-sT)\text{d}N_{\rho}^{\lambda}(T).
\end{equation}
Alternatively, we can find another expression of $\eta_{\rho}^{\lambda}$ in terms of the random transfer operator. In fact,
$$\eta_{\rho}^{\lambda}(s)=\sum _{n=1} ^{\infty} \exp(-sT_i)\left(N_{\rho}^{\lambda}(T_i)-N_{\rho}^{\lambda}(T_{i-1})\right),$$
where $T_1 < T_2 < T_3 <...$ is the increasing sequence of discontinuities of $N_{\rho}^{\lambda}(T)$. The random Poincar\'{e} series can also be expressed by 
$$\eta_{\rho}^{\lambda}(s)=\sum_{n=1} ^{\infty} \sum\limits_{\omega \in E_{\rho}^n}\exp(sS_nf(\omega \rho,\lambda))$$
$$\hspace{.6cm} =\sum_{n=1} ^{\infty} \sum\limits_{\omega \in E_{\rho}^n}\exp(sS_nf(\omega \rho,\lambda))=\sum_{n=1} ^{\infty} (\mathcal{L}_s^{\lambda}) ^n (\mathbb{1})(\rho)=\sum_{n=1} ^{\infty}|\lambda_0\lambda_1...\lambda_{n-1}|^s\mathcal{L}_s^n(\mathbbm{1})(\rho).$$ 
Let $\delta^{\lambda}$ represent the \textbf{critical line of convergence} for this series.

\noindent This assumption enables us to obtain another expression of this random Poincar\'{e} series. In fact, by the spectral decomposition from \ref{spectral decomposition sum}, we can write
$$\eta_{\rho}^{\lambda}(s)=\sum_{k=1} ^{\infty}|\lambda_0\lambda_1...\lambda_{k-1}|^s\mathcal{L}_s^k(\mathbbm{1})(\rho)$$
$$=\sum _{k=1} ^{\infty}|\lambda_0\lambda_1...\lambda_{k-1}|^s \left(\xi_1^k(s) \mathcal{P}_{1,s}(\mathbbm{1})+ \xi_2^k(s)\mathcal{P}_{2,s}(\mathbbm{1}) + ...+ \xi_n^k(s)P_{n,s}(\mathbbm{1})+\mathcal{D}_s^k(\mathbbm{1})\right)$$
\begin{equation}\label{random spectral decom}
=\left(\sum _{k=1} ^{\infty}|\lambda_0\lambda_1...\lambda_{k-1}|^s \xi_1^k(s)\right) \mathcal{P}_{1,s}(\mathbbm{1})+...+\left(\sum _{k=1} ^{\infty}|\lambda_0\lambda_1...\lambda_{k-1}|^s \xi_n^k(s)\right) \mathcal{P}_{n,s}(\mathbbm{1})
\end{equation}
$$+\sum _{k=1} ^{\infty}|\lambda_0\lambda_1...\lambda_{k-1}|^s \mathcal{D}_s^k(\mathbbm{1}).$$
It is important to note that the analogue of proposition 13 in \cite{HN} is not easy to construct for random systems as the transfer operator under consideration here is a random operator and the conventional notion of eigenvalue does not exist here for us in this section. Therefore to find an asymptotic formula for the counting function \ref{random counting}, we have to apply Ikehara-Wiener theorem or Garaham-Vaaler theorem directly, see theorems 3,4 in \cite{HN}. This requires understanding the behavior of each of parenthesis in \ref{random spectral decom} above, along $x=\delta^{\lambda}$.

\section{Exponential Growth}\label{exp growth}

\noindent We remind that we are only working with finitely irreducible strongly regular systems in this paper.
\begin{definition}\label{future priodic}
We say $\lambda \in \Lambda$ is \textbf{future periodic}, if there exists $k \in \mathbb{N}$ such that $ \sigma^k(\lambda_0\lambda_1...)=\lambda_0\lambda_1...,$
and we say it is \textbf{ eventually future periodic}, if there exist $k \in \mathbb{N}$ and an integer $m \geq 0$ such that $ \sigma^k(\lambda_m\lambda_{m+1}...)=\lambda_m\lambda_{m+1}... $ .
\end{definition}
\begin{theorem}\label{future periodic d-generic}
Given a random CIFS constructed in section \ref{sec 3} with D-generic property and each eventually future periodic $\lambda$. There exist $C,D>0$, such that 
$$C\leq \liminf_{T \rightarrow \infty}\dfrac{N_{\rho}^{\lambda}(T)}{e^{\delta^{\lambda}T}} \leq \limsup_{T \rightarrow \infty}\dfrac{N_{\rho}^{\lambda}(T)}{e^{\delta^{\lambda}T}}\leq D.$$
\end{theorem}
\begin{proof}
We let 
$$\lambda=...\lambda_m\lambda_{m+1}\lambda_{m+2}...\lambda_{m+k-1}\lambda_m\lambda_{m+1}\lambda_{m+2}...\lambda_{m+k-1}...,$$
$$a_i(s):=|\lambda_m...\lambda_{m+k-1}|^s\xi_i(s)^{k},$$
$$c_i(s):=|\lambda_0...\lambda_{m-1}|^s\left(|\lambda_m|^s\xi_i(s)+|\lambda_m\lambda_{m+1}|^s\xi_i(s)^2+...+|\lambda_m...\lambda_{m+k-1}|^s\xi_i(s)^k\right).$$
Then there exists a holomorphic function $b_i(s)$ such that the $i^{th}$ parenthesis in \ref{random spectral decom} above, can be written as:
$$b_i(s)+a_i(s)c_i(s)+a_i(s)^2c_i(s)+a_i(s)^3c_i(s)+...=b_i(s)+c_i(s)\dfrac{a_i(s)}{1-a_i(s)}.$$
Now since $a_i(s)$ is holomorphic function, there should be $y_i>0$ such that $a_i(s)$ omits $1$ on $\{s=\delta^{\lambda}+iy: -y_i<y<y_i, \;  y\neq 0\}$, unless $\xi_i(s)=(\sqrt[k]{|\lambda_m...\lambda_{m+k-1}|})^{-s}$, which implies that $i=1$. But in this case $\xi_1(s)$ meets $1$ infinitely often which cannot happen as we assumed D-generic property. This means there is $y_0>0$ such that for each $i$, the function $a_i$ omits $1$ on $\{s=\delta^{\lambda}+iy: -y_0<y<y_0, \; y\neq 0\}$. Therefore the $i^{th}$ parenthesis has a continuous extension at least on this segment and so Graham-Vaaler theorem (see theorem 4 in \cite{HN}) is applicable. This finishes the proof.
\end{proof}
Note that the constants $C,D$ depend on $\lambda$. In fact, one can construct a sequence of eventually future periodic $\lambda$ for which $C \to 0$ and $D \to \infty$. Also, it is easy to see that $\delta^{\lambda}$ is the unique root of 
$$P(x)+\frac{x}{k}\log |\lambda_m...\lambda_{m+k-1}|=0.$$
Additionally, it is clear that for many future periodic $\lambda$, we have $\delta_{\lambda}\neq \delta_{\Lambda}$. Therefore $$N_{\rho}^{\lambda}(T) \not\sim \exp(\delta_{\Lambda} T), \; \; \; \text{as } \; \; T \to \infty.$$
If further the potential of the deterministic system $\mathcal{S}=\{ \phi_{e} \}_{e \in E}$ is assumed to be strongly D-generic, we can improve the above theorem.
\begin{corollary}\label{future periodic strongly d-generic}
Given a random CIFS constructed in section \ref{sec 3} with strongly D-generic property and each eventually future periodic $\lambda$. There is a constant $C>0$ such that 
$$ \lim_{T \rightarrow \infty}\dfrac{N_{\rho}^{\lambda}(T)}{e^{\delta^{\lambda}T}}=C.$$
\end{corollary}
\begin{proof}
Along the proof of the previous theorem, it is enough to notice that $a_i(s)$ can meet $1$ only if $i=1$ and $s$ is real, which means all along the critical line of convergence (except at the real point) we have a continuous extension and therefore Ikehara-Wiener theorem (see theorem 3 in \cite{HN})  is applicable in this case. 
\end{proof}
\noindent Below we discuss some notions from the theory of random walk to prove some results for our counting problem. For simplicity, we restrict ourselves to the case when $\mathcal{Z}$ is finite. With some special care, one can generalize to infinite $\mathcal{Z}$. For each positive integer $n$, each $z \in \mathcal{Z}$ and every $\lambda \in \Lambda$ we let $s_{n,z}(\lambda)$ denote the number of times that $z$ appears in $\lambda_0...\lambda_{n-1}$. It is obvious that for each $\lambda \in \Lambda$ and $n \in \mathbb{N}$ we have
$$\sum_{z \in \mathcal{Z}}s_{n,z}(\lambda)=n.$$
Assume $\#\mathcal{Z}=a$ and $\nu([\lambda_0=z_i])=p_i$, then we can assign a random walk of the following form
$$s_n=(s_{n,z_2}, s_{n,z_3}, ..., s_{n,z_{a}}),$$
which is a $(a-1)$-dimensional random walk, because $s_{n,z_1}$-coordinate is determined by $s_{n,z_1}=n-\sum_{i=2}^a s_{n,z_i}$. We know the normalized random walk 
$$s_n-n(p_2, p_2, ..., p_a)$$
will be recurrent almost surely when $a\leq 3$ by a theorem of Pólya in probability theory, see \citep[193]{revesz}.
\begin{definition}\label{bounded boundary}
We say $\lambda \in \Lambda$ is of \textbf{bounded fluctuation} if there exists a vector $l=(l_2, l_3, ..., l_a)$ such that the assigned walk $s_n-n(l_2, l_2, ..., l_a)$ is bounded forever. Alternatively, $\lambda \in \Lambda$ is of bounded fluctuation if there are numbers $p,q \in \mathbb{R}$ and $l_{2}, ..., l_a \in \mathbb{R}^{\geq 0}$ such that for all for all $2 \leq i \leq a$ and all $n \in \mathbb{N}$: 
$$n l_i +p\leq s_{n,z_i}(\lambda) \leq n l_i +q.$$
\end{definition}
Note that in the above definition, it does not make sense to allow $l_i$ be negative. Also, if the number of times that the letter $z_i \in \mathcal{Z}$ appears in $\lambda_0\lambda_1\lambda_2...$ is finite then $l_{i}=0$. Also $l_i$'s are subject to $\sum_{i=1}^n l_i=1$ because we said $\sum_z s_{n,z}=n$. 
\noindent To prove the following theorem we set
$$c:=\sum_{i=1}^a l_i\log|z_i|, \; \; \; \;  d:=\sum_{i=1}^a \log|z_i|.$$
We remind that $f$ is the potential function for the deterministic system $\mathcal{S}=\{ \phi_{e} \}_{e \in E}$.
\begin{theorem}\label{main theorem 5}
Given a random CIFS constructed in section \ref{sec 3}. We assume $g:=f+c$ inherits strong regularity from $f$ (the potential function of the deterministic system) and $\#\mathcal{Z}$ is finite. For any $\lambda$ of bounded fluctuation, there exist constants $C,D>0$ such that 
$$C\leq \liminf_{T \rightarrow \infty}\dfrac{N_{\rho}^{\lambda}(T)}{e^{\delta^{\lambda}T}} \leq \limsup_{T \rightarrow \infty}\dfrac{N_{\rho}^{\lambda}(T)}{e^{\delta^{\lambda}T}}\leq D.$$
\end{theorem}
\begin{proof}
By the definition of bounded fluctuation, we can find the following inequalities for the random ergodic sum: 
\begin{align*}
S_nf(\omega\rho)+nc+qd & \leq S_nf(\omega\rho,\lambda)=S_nf(\omega\rho)+\log|\lambda_0...\lambda_{n-1}| \\
 & \leq S_nf(\omega\rho)+nc+pd
\end{align*}
Next, we consider the function $g=f+c$, and we assign a counting function for it, say $N_{\rho}'(T)$. Then using the above inequalities we get
$$N'_{\rho}(T+q d) \leq N_{\rho}^{\lambda}(T) \leq N'_{\rho}(T+p d).$$
This yields
$$
\exp(\delta^{\lambda}q d) \dfrac{N'_{\rho}(T+q d)}{\exp(\delta^{\lambda}(T+q d))} \leq \dfrac{N_{\rho}^{\lambda}(T)}{\exp(\delta^{\lambda}T)} \leq \dfrac{N'_{\rho}(T+p d)}{\exp(\delta^{\lambda}(T+p d))}\exp(\delta^{\lambda}p d).
$$
It is enough to use theorem \ref{border 0 non-degeneric} to conclude the proof.
\end{proof}

\noindent Note that when $E$ has a finite number of alphabets then $g$ inherits strong regularity from $f$, therefore the first assumption in this theorem is redundant. Furthermore, it is clear that every eventually periodic $\lambda$ is of bounded fluctuation. Therefore this theorem implies theorem \ref{future periodic d-generic} when $\#\mathcal{Z}$ is finite. However, it doesn't imply corollary \ref{future periodic strongly d-generic}. It is essential to note that we did not assume the D-generic property in theorem \ref{main theorem 5}. Additionally, it is not hard to see that $\delta^{\lambda}$ is the unique root of 
$$P(x)+cx=0.$$
Note that if $f$ also has the D-generic property, then the bounds obtained above may not necessarily be improved. However, by imposing strongly D-generic property on $f$ we get a better estimate for $C$ and $D$. For this, it is enough for us to notice 
$$\mathcal{L}_{sg}=e^{sc}\mathcal{L}_{sf},$$
$$P(xg)=P(xf)+cx,$$
therefore in case $\mathcal{L}_{sg}$ admits $e^{P(xf)+cx}$ as an eigenvalue, $\mathcal{L}_{sf}$ must admit $e^{P(xf)-ciy}$ as an eigenvalue, but this cannot happen, since $\mathcal{L}_{sf}$ does not admit any eigenvalue of modulus $e^{P(x)}$.

\section{Non-Exponential Growth} \label{non exp growth}

\noindent So far the classes of $\lambda \in \Lambda$ that we investigated the counting function for, have all $\nu$-measure 0. In fact, the set of eventually periodic $\lambda \in \Lambda$ or the set of $\lambda \in \Lambda$ with bounded fluctuation have $\nu$-measure 0. We noticed that the counting problem in random dynamics had similar behavior to that of deterministic systems under certain conditions, i.e. the growth rate of the counting function is exponential. However, this behavior does not hold in general. In fact, for most $\lambda$ this behavior is not observed. We provide such a system below in which for almost every $\lambda \in \Lambda$ the exponential growth is not observed. 

\noindent We let 
$$\mathcal{Z}=\{ z, w \}, \; \; \; \; \; 0<z<w<1.$$ 
Therefore,
$$\Lambda=\{z,w\}^{\mathbb{Z}}.$$
We consider the Bernoulli measure on $\Lambda$,
$$\nu([\lambda_0=z])=p, \; \; \; \; \nu([\lambda_0=w])=q, \;  \; \; \; \; \; p+q=1.$$
Next, we set $E:=\{ 0,1\}$ and we consider any arbitrary irreducible incidence matrix $A$. For the deterministic CIFS, we consider
$$\phi_e(t)=\alpha_e t+\beta_e, \; \; \; \; \; \;t\in [0,1], \; \; \alpha_e, \beta_e \in (0,1), \; \; \; e \in E,$$
where $\alpha_e,\beta_e$'s are subject to the conditions of a conformal graph-directed Markov system, see definition 15 in \cite{HN}.
Following section \ref{sec 3}, we construct a random CIFS below
$$\phi_e^{\lambda}(t)=\lambda_0\phi_e(t)=\lambda_0(\alpha_e t+\beta_e), \; \; \; \; \; e \in E, \; \; \lambda \in \Lambda.$$
Then one can find the random potential:
$$f(\rho, \lambda)=\log |(\phi_{\rho_1}^{\lambda})'(\pi^{T(\lambda)}(\sigma\rho))|=\log (\alpha_{\rho_1} \lambda_0), \; \; \; \; \rho \in E^{\mathbb{N}}, \; \; \lambda \in \Lambda,$$
and the random pressure:
$$P^{\lambda}(x)=P(x)+x\log \lambda_0=\log (\alpha_0^x+\alpha_1^x)+x\log ( \lambda_0), \; \; \; \; \; x \in \mathbb{R}, \; \; \lambda \in \Lambda.$$ 
Therefore, the expected pressure is given by:
$$\mathcal{E}P(x)=\int_{\Lambda}P^{\lambda}(x)\text{d}\lambda=\log (\alpha_0^x+\alpha_1^x)+x(p\log z+q\log w), \; \; \; x \in \mathbb{R}.$$
We also want to use the law of the iterated logarithm \ref{lil} in this setting. For the random walk $s_{n,z}$, we obtain
\begin{equation}\label{lil2}
    \liminf_{n \to \infty} \dfrac{s_{n,z}(\lambda)-np}{\sqrt{2pq n\log \log n}}=-1, \; \; \; \; \; \limsup_{n \to \infty} \dfrac{s_{n,z}(\lambda)-np}{\sqrt{2pq n\log \log n}}=1 \; \; \; \; \; \text{a.e. } \lambda \in \Lambda
\end{equation}
Note that this law implies the fluctuation of the normalized random walk 
$$s_{n,z}(\lambda)-np$$
in the sense that every state is met infinitely often since it is a one-dimensional random walk, see \citep[193]{revesz}.

\noindent We mention below the inequality obtained in the example 3 in \cite{HN} (when $\alpha_0=\alpha_1=\alpha$)
\begin{equation}\label{eqr 77}
C \leq \liminf_T \dfrac{N (T)}{\exp(\delta T)}\leq \limsup_T \dfrac{N (T)}{\exp(\delta T)} \leq D,
\end{equation}
where $C,D>0$ and $\delta=\log r(A)/-\log \alpha$.
\begin{theorem}\label{main theorem 4}
For the random CIFS constructed above, when $\alpha_0=\alpha_1=\alpha$, we have
$$\liminf_T \dfrac{N^{\lambda}_{\rho}(T)}{\exp(\delta_{\Lambda} T)}=0, \; \; \; \limsup_T \dfrac{N^{\lambda}_{\rho}(T)}{\exp(\delta_{\Lambda} T)}=\infty, \; \; \; \nu-\text{a.e.} \; \lambda \in \Lambda.$$
\end{theorem}
\begin{proof}
It is not hard to see that the root of the expected pressure is
$$\delta_{\Lambda}=-\frac{\log r(A)}{\log \alpha+p\log z+q\log w}.$$
Now we consider a $\lambda \in \Lambda$ for which the law of the iterated logarithm \ref{lil2} holds. Due to the fluctuation of random walk, for an arbitrary integer $m$, one can find a sequence $n_i$ such that 
$$s_{n_i,z}({\lambda}) \leq pn_i+m < s_{n_i+1,z}(\lambda).$$ 
Therefore, we can give the following estimate:
\begin{equation}\label{eqr 9}
n_i\log \alpha+(pn_i+m)\log z +(qn_i-m+1)\log w  \leq S_{n_i}f(\omega \rho,\lambda)
\end{equation}
$$\leq n_i\log \alpha+(pn_i+m-1)\log z +(qn_i-m)\log w .$$
Furthermore, note that since $f(\rho, \lambda)$ and so $S_n(\rho,\lambda)$ are constant negative functions in $\rho$ variable, then for $T_i=-S_{n_i}f(\omega \rho, \lambda)$, we obtain the formula:
$$N_{\rho}^{\lambda} (T_i)=\# \{\omega \in E_{\rho}^* : S_{|\omega|}f(\omega \rho,\lambda)\geq -T_i\}=\# \{\omega \in E_{\rho}^* : |\omega|\leq n_i\}$$
$$=\# \{\omega \in E_{\rho}^* : |\omega|\log \alpha \geq n_i\log \alpha\}=\# \{\omega \in E_{\rho}^* :  S_{|\omega|}f(\omega \rho) \geq n_i\log \alpha\}$$
$$=N_{\rho}(-n_i\log \alpha).$$
This formula relates the counting problem in the random setting to the deterministic setting. Therefore, we can use the counting formula for the deterministic setting expressed in the inequality $\ref{eqr 77}$ above. For any small enough $\epsilon>0$, there exist constants $C,D>0$ and an integer $N$ such that for $i\geq N$:
\begin{equation}\label{recent}
 (C-\epsilon)r(A)^{n_i} \leq N_{\rho}^{\lambda} (T_i)\leq (D+\epsilon)r(A)^{n_i}.
\end{equation}
Additionally, inequality $\ref{eqr 9}$ gives:
$$\exp\left( -\delta_{\Lambda}\left( n_i\log \alpha+(pn_i+m-1)\log z +(qn_i-m)\log w\right)\right)$$
$$\hspace{-4.3cm} \leq \exp(\delta_{\Lambda}T_i)=\exp (-\delta_{\Lambda}S_{n_i}f(\omega \rho, \lambda))$$
$$\leq \exp \left(-\delta_{\Lambda}\left( n_i\log \alpha+(pn_i+m)\log z +(qn_i-m+1)\log w\right)\right),$$
which can be rewritten as
$$r(A)^{n_i}\exp\left(\delta_{\Lambda}(m\log \frac{w}{z}+\log z)\right) \leq \exp(\delta_{\Lambda}T_i)\leq r(A)^{n_i} \exp\left(\delta_{\Lambda}(m\log \frac{w}{z}-\log w)\right).$$
This along with \ref{recent} yields:
$$ (C-\epsilon) \exp\left(-\delta_{\Lambda}(m\log \frac{w}{z}-\log w)\right)\exp(\delta_{\Lambda}T_i) \leq N_{\rho}^{\lambda} (T_i) $$
$$\leq (D+\epsilon)\exp\left(-\delta_{\Lambda}(m\log \frac{w}{z}+\log z)\right) \exp(\delta_{\Lambda}T_i).$$
By passing to a subsequence, we find: 
$$(C-\epsilon)w^{\delta_{\Lambda}}(\frac{z}{w})^{\delta_{\Lambda}m} \leq \lim_{i} \frac{N_{\rho}^{\lambda} (T_i)}{ \exp(\delta_{\Lambda}T_i)} \leq \frac{D+\epsilon}{z^{\delta_{\Lambda}}}(\frac{z}{w})^{\delta_{\Lambda}m}.$$
Letting $m\to \infty$, we get:
$$\liminf_T  \frac{N_{\rho}^{\lambda} (T)}{ \exp(\delta_{\Lambda}T)}=0,$$
and if $m\to -\infty$, we obtain:
$$\limsup_T  \frac{N_{\rho}^{\lambda} (T)}{ \exp(\delta_{\Lambda}T)}=\infty.$$
\end{proof}
\noindent This theorem tells us the counting function does not have exponential growth for random systems in general. In fact, for almost all $\lambda \in \Lambda$, we cannot find constants $a^{\lambda},C^{\lambda},D^{\lambda},T^{\lambda}>0$ such that 
$$C^{\lambda} \exp(a^{\lambda}T) \leq N_{\rho}^{\lambda} (T) \leq D^{\lambda} \exp(a^{\lambda}T) \; \; \; \; T \geq T^{\lambda}.$$
This means that  Graham-Vaaler's Tauberian theorem (see theorem 4 in \cite{HN}) is not applicable to finding a general formula for all random systems. Actually, the boundary behavior of the random Poincar\'{e} series $\eta_{\rho}^{\lambda}(s)$ in \ref{random spectral decom} is not similar to that of the deterministic systems, in the sense that the complex function
$$\eta_{\rho}^{\lambda}(s)-\frac{A}{s-\delta_{\Lambda}}$$
does not have a nice bounded behavior on any vertical interval centered at $s=\delta_{\Lambda}$ on the critical line $x=\delta_{\Lambda}$. We cannot claim anything in general about the behavior of the random Poincar\'{e} series near the point $s=\delta_{\Lambda}$. There are examples of random systems for which the random Poincar\'{e} series is continuous on the critical line $x=\delta_{\Lambda}$ except at countably many (discrete) points one of which  $s=\delta_{\Lambda}$. However, we do not know about the singularity type of $s=\delta_{\Lambda}$. The best scenario is that it is a simple pole (on the right half plane $\Gamma^+$). In this case, the growth of the counting function will be exponential. One theorem from Tauberian theory suggests that the behavior of the Poincar\'{e} series may be essential, see theorem 2.1 in \cite{Ryo}. Because if at $s=\delta_{\Lambda}$ we have a pole of any order, then theorem \ref{main theorem 4} should not hold.

\noindent We add that the above theorem was proven under the assumption $\# \mathcal{Z}=2$. One may ask for other cases. Note that the random walk $s_{n,z}$ above is a one-dimensional random walk. When we have more elements in $\mathcal{Z}$, the dimensional of the random walk also increases. As stated earlier, if $\#\mathcal{Z}=a$ and $\nu([\lambda_0=z_i])=p_i$, then the random walk will be of the form
$$s_n=(s_{n,z_2}, s_{n,z_3}, ..., s_{n,z_{a}}),$$
which is a $(a-1)$-dimensional random walk. When $a\leq 3$ we can still establish a similar result. Also when $a=4$ one can still claim a similar result with some effort because the projection of the 3-dimensional random walk $s_n$ on the hyperplane with the normal $(p_2, p_3, p_4)$ is a 2-dimensional random walk and therefore Pólya theorem is applicable, see \citep[193]{HN}. But for any $a>4$, one requires a deeper understanding of the theory of random walk. To find the answer, it is important to know which ones of $(s_{n,z_i}-np_i)$'s are positive and which ones are negative. More precisely, when the sum
$$\sum_{i=1}^a (s_{n,z_i}-np_i)\log |z_i| \to \infty, \; \; \; \text{as} \; \;  n \to \infty,$$
then the counting function grows faster than the exponential and so the limit supremum of $\dfrac{N^{\lambda}(T)}{\exp(\lambda T)}$ will be infinity. On the other hand, when 
$$\sum_{i=1}^a (s_{n,z_i}-np_i)\log |z_i| \to -\infty, \; \; \; \text{as} \; \; n \to \infty,$$
then the counting function grows slower than the exponential and so the limit infimum of $\dfrac{N^{\lambda}(T)}{\exp(\lambda T)}$ will be zero. We conjecture that this phenomenon is occuring for almost all $\lambda \in \Lambda$ even when $a>4$.

\section{Examples}
\begin{example}
For the random CIFS constructed at the beginning of section \ref{non exp growth}, we set
$$\alpha=\frac{1}{3}, z=\frac{1}{5}, w=\frac{1}{7}.$$
Then it is not hard to see that we can actually get a bit stronger result than that of the theorem \ref{main theorem 4}. In fact, the limit set of $\dfrac{N^{\lambda}_{\rho}(T)}{\exp(\delta_{\Lambda} T)}$ as $T \to \infty$ is $[0,\infty]$, i.e. for each nonnegative real $C$ we can find a sequence $T_n \to \infty$ such that 
$$ \lim_{n \to \infty} \dfrac{N^{\lambda}_{\rho}(T_n)}{\exp(\delta_{\Lambda} T_n)}=C.$$
This is the case for $\nu$-a.e. $\lambda \in \Lambda$.
\end{example}
\begin{example}
If we consider a Schottky group that generates the Apollonian gasket as a deterministic system, it is known that with some modifications of corollary \ref{llll} one obtains a polynomial growth formula for the number of circles of radius at least $1/T$ in the packing, see theorem 6.2.13 in \cite{UP}. However, investigating the counting problem for the random Schottky group is not an easy question in general. But, it is good to notice that if $\mathcal{Z}$ is a subset of the unit circle, the answer would be exactly the same as that of the deterministic Schottky group due to the fact that the random ergodic sum is identical to the deterministic ergodic sum, see \ref{random ergodic sum}. This is actually expected, since each $z \in \mathcal{Z}$ may change each limit point, but it leaves the circles of inversions intact. In fact, these random factors act as rotations.
\end{example}
\noindent We just want to emphasize that the above example falls in the category of random CGDMS since the Schottky group does not arise from a CIFS. But still, our results are applicable as long as the very definition of random system is met.
\begin{example}\label{2,3,5,7}
Given the random CIFS constructed at the beginning of section \ref{non exp growth}, we set
$$\alpha_0=\dfrac{1}{2}, \; \; \; \alpha_1=\dfrac{1}{3}, \; \; \; z=\frac{1}{5}, \; \; \; w=\frac{1}{7}.$$
This is an example of a system that has the D-generic property which does not have the strongly D-generic property. The random Poincar\'{e} series for a future periodic $\lambda$ (given that $\sigma^m(\lambda_0\lambda_1\lambda_2...)=\lambda_0\lambda_1\lambda_2...$ for some $m \in \mathbb{N}$) can be simplied as (see proof of theorem \ref{future periodic d-generic}):
$$\eta_{\rho}^{\lambda}(s)=\left((\frac{1}{2^s}+\frac{1}{3^s})\lambda_0^s+(\frac{1}{2^s}+\frac{1}{3^s})^2\lambda_0^s\lambda_1^s+...+(\frac{1}{2^s}+\frac{1}{3^s})^m\lambda_0^s...\lambda_{m-1}^s\right)
\theta(s)^{-1}, \; \; \; x>x_0$$
where 
$$\theta(s)=1-(\frac{1}{2^s}+\frac{1}{3^s})^m\lambda_0^s...\lambda_{m-1}^s,$$
and $x_0$ is a real root of $\theta(s)$.
Using the triangle inequality we can see $x_0$ is actually the only root of $\theta(s)$ along $x=x_0$, see example 5 in \cite{HN} for a similar argument. We want to show $x_0$ is a root of multiplicity one. Note
$$\theta(s)=1-(\frac{1}{2^s}+\frac{1}{3^s})^m\lambda_0^s...\lambda_{m-1}^s=1-\left((\frac{\sqrt[m]{\lambda_0...\lambda_{m-1}}}{2})^s+(\frac{\sqrt[m]{\lambda_0...\lambda_{m-1}}}{3})^s\right)^m,$$
so by taking derivative we get
$$\theta'(s)=-m\left(\log \frac{\sqrt[m]{\lambda_0...\lambda_{m-1}}}{2}(\frac{\sqrt[m]{\lambda_0...\lambda_{m-1}}}{2})^s+\log \frac{\sqrt[m]{\lambda_0...\lambda_{m-1}}}{3}(\frac{\sqrt[m]{\lambda_0...\lambda_{m-1}}}{3})^s\right).$$
$$\times \left((\frac{\sqrt[m]{\lambda_0...\lambda_{m-1}}}{2})^s+(\frac{\sqrt[m]{\lambda_0...\lambda_{m-1}}}{3})^s\right)^{m-1}$$
If $\theta$ has a root of multiplicity 2 or higher, then $\theta'$ as well vanishes at that root. Looking for such $s$ we should have
\begin{equation}\label{r11}
\left((\frac{\sqrt[m]{\lambda_0...\lambda_{m-1}}}{2})^s+(\frac{\sqrt[m]{\lambda_0...\lambda_{m-1}}}{3})^s\right)^m=1,
\end{equation}
$$\log \frac{\sqrt[m]{\lambda_0...\lambda_{m-1}}}{2}(\frac{\sqrt[m]{\lambda_0...\lambda_{m-1}}}{2})^s+\log \frac{\sqrt[m]{\lambda_0...\lambda_{m-1}}}{3}(\frac{\sqrt[m]{\lambda_0...\lambda_{m-1}}}{3})^s=0.$$
The latter one leaves
\begin{equation}\label{r12}
(\frac{2}{3})^s=-\frac{\log \sqrt[m]{\lambda_0...\lambda_{m-1}}-\log 2}{\log \sqrt[m]{\lambda_0...\lambda_{m-1}}-\log 3}.
\end{equation}
Let $\beta$ represent the right-hand side of the above equality. Note that since $\beta$ is real, then the left-hand side must be real either. This gives:
$$y=\frac{j\pi}{\log(2/3)}, \; \; j \in \mathbb{Z}.$$
Substituting \ref{r12} into \ref{r11}, yields:
$$\left((\frac{\sqrt[m]{\lambda_0...\lambda_{m-1}}}{2})^s+\beta(\frac{\sqrt[m]{\lambda_0...\lambda_{m-1}}}{2})^s\right)^m=1.$$
Therefore,
$$(\frac{\sqrt[m]{\lambda_0...\lambda_{m-1}}}{2})^{ms}=\frac{1}{(1+\beta)^m}.$$
Again the right-hand side of this is real, and so
$$y=l\pi(m\log\frac{\sqrt[m]{\lambda_0...\lambda_{m-1}}}{2})^{-1}, \; \; \; l \in \mathbb{Z}.$$
We found two expressions for $y$. This is possible only if $\log(\frac{2}{3})/\log (\frac{\sqrt[m]{\lambda_0...\lambda_{m-1}}}{2})$ is rational, which cannot happen. This shows all the roots of $\theta$ are of multiplicity one and so all the poles of the random Poincar\'{e} series are simple.
Therefore the Ikehara-Wiener theorem is applicable, see theorem 3 in \cite{HN}. Thus, there exists $C>0$ such that
$$\frac{N_{\rho}^{\lambda}(T)}{\exp(\delta_{\Lambda}T)} \rightarrow C, \; \; \text{as} \; \; T \to \infty.$$
\end{example}

\noindent It is clear that we can find a similar formula for eventually future periodic $\lambda$ in the previous example. Furthermore, as we said the above system is not strongly D-generic; however, the counting function has similar asymptotic growth to that of the systems with strongly D-generic property, see corollary \ref{future periodic strongly d-generic}.

\noindent At the end, we want to make a few remarks. One is that the reason we focused on random CIFS in this paper is mostly related to the very definition of the random system. We observed that our counting methods highly depend on the fact that the deterministic ergodic sum can be separated in the expression of the random ergodic sum. This allowed us to employ the well-known methods for the counting problem combined with the theory of random walk. When the deterministic ergodic sum cannot be separated in the random ergodic sum, we face a much harder problem. This depends on how one constructs a random system. Second, we want to emphasize that our results can also be expressed for random CGDMS as long as the very definition of the random system is satisfied. In fact, we limited ourselves to the random CIFS since when we multiply the deterministic function systems $\{\phi_e \}$ (all defined on a compact Euclidean space $X$) by a random factor $\lambda_0$. This makes the random function system $\{ \phi_e^{\lambda} \}$ well-defined on the same compact Euclidean space $X$. When one goes beyond this to the random CGDMS, one should be careful with multiplying by a random factor. But as long as the very definition of the random system is satisfied, the aforementioned results will also carry through. Third, one may be interested in a more inclusive result that contains all random systems. We conjecture that such a result can only have a probabilistic form. In fact, if one wants to find a formula that captures information on the growth rate of any given random counting function, they may want to investigate a formula that involves some limit theorem on $\Lambda$. The alternative approach, but of less hope and in some special cases, would be understanding the General Dirichlet series. The well-known Dirichlet series does not help much in studying random systems. Finally, we emphasize that in this paper we did not focus on counting periodic words. We gave a thorough estimate for approximating the "counting function for periodic words" by the "counting function for finite words" in \cite{HN}. Therefore, it is clear that once the latter has exponential growth the former also resembles similar behavior.

\end{large}

\bibliographystyle{alpha}
\bibliography{bibbb.bib}
\end{document}